\documentclass{amsart}%
\usepackage{amsfonts}
\usepackage{amsmath}
\usepackage{amssymb}
\usepackage[unicode=true]{hyperref}
\usepackage{graphicx}%
\providecommand{\U}[1]{\protect\rule{.1in}{.1in}}
\newtheorem{theorem}{Theorem}
\theoremstyle{plain}

\newtheorem{corollary}{Corollary}

\newtheorem{lemma}{Lemma}

\newtheorem{remark}{Remark}

\numberwithin{equation}{section}
\begin{document}
\title[John-Nirenberg
self improving inequalities]{Marcinkiewicz spaces, Garsia-Rodemich spaces and the scale of John-Nirenberg
self improving inequalities}
\author{Mario Milman}
\address{Instituto Argentino de Matematica}
\email{mario.milman@gmail.com}
\urladdr{https://sites.google.com/site/mariomilman}
\thanks{The author was partially supported by a grant from the Simons Foundation
(\#207929 to Mario Milman)}

\begin{abstract}
We extend to n-dimensions a characterization of the Marcinkiewicz
$L(p,\infty)$ spaces first obtained by Garsia-Rodemich in the one dimensional
case. This leads to a new proof of the John-Nirenberg self-improving
inequalities. We also show a related result that provides still a new
characterization of the $L(p,\infty)$ spaces in terms of distribution
functions, reflects the self-improving inequalities directly, and also
characterizes $L(\infty,\infty),$ the rearrangement invariant hull of $BMO.$
We show an application to the study of tensor products with $L(\infty,\infty)$
spaces, which complements the classical work of O'Neil \cite{oneil} and the
more recent work of Astashkin \cite{astashkin}.

\end{abstract}
\maketitle

\section{Introduction}

In their seminal paper \cite{jn}, John-Nirenberg introduced the space $BMO$
and proved the celebrated John-Nirenberg inequality for functions in $BMO.$ It
is also well known, although perhaps somewhat less so, that in the same paper,
John-Nirenberg showed that the $BMO$ self improvement inequality can be
refined and framed as a scale of inequalities. These inequalities (or
embeddings) are associated with what we nowadays call \textquotedblleft
John-Nirenberg spaces". This result of John-Nirenberg, which we now describe,
is the starting point of our development in this paper.

Let $Q_{0}\subset\mathbb{R}^{n},$ be a fixed cube\footnote{A \textquotedblleft
cube" in this paper will always mean a cube with sides parallel to the
coordinate axes.}$,$ $1\leq p<\infty.$ Let
\begin{align*}
P(Q_{0})  &  =\{\{Q_{i}\}_{i\in N}:\text{countable families of subcubes }%
Q_{i}\subset Q_{0},\text{ }\\
&  \text{with pairwise disjoint interiors}\}.
\end{align*}
The John-Nirenberg spaces are defined by%
\[
JN_{p}(Q_{0})=\{f\in L^{1}(Q_{0}):JN_{p}(f,Q_{0})<\infty\},
\]
where\footnote{in what follows, as usual, $f_{Q}=\frac{1}{\left\vert
Q\right\vert }\int_{Q}fdx.$}%
\[
JN_{p}(f,Q_{0})=\sup_{\{Q_{i}\}_{i}\in P(Q_{0})}\left\{  \left\{
{\displaystyle\sum\limits_{i}}
\left\vert Q_{i}\right\vert \left(  \frac{1}{\left\vert Q_{i}\right\vert }%
\int_{Q_{i}}\left\vert f-f_{Q_{i}}\right\vert dx\right)  ^{p}\right\}
^{1/p}\right\}  .
\]
Let us also recall that, for a given measure space, the Marcinkiewicz
$L(p,\infty)$ spaces, $1\leq p<\infty,$ are defined by demanding\footnote{Here
$f^{\ast}$ denotes the non-increasing rearrangement of $f$ and $\lambda_{f}$
its distribution function (cf. \cite{bs}).} that $\left\Vert f\right\Vert
_{L(p,\infty)}^{\ast}<\infty,$ where%
\begin{equation}
\left\Vert f\right\Vert _{L(p,\infty)}^{\ast}=\sup_{t>0}\{f^{\ast}%
(t)t^{1/p}\}=\sup_{t>0}\{t\left(  \lambda_{f}(t)\right)  ^{1/p}\};
\label{verarriba}%
\end{equation}
while for $p=\infty,$ the space $L(\infty,\infty)$ (cf. \cite{bds}) is
defined\footnote{Some authors (including sometimes the author of this paper)
use a different notation and let $W$ denote what we call $L(\infty,\infty),$
at the same time that they use the notation $L(\infty,\infty)=L^{\infty}.$} by
the condition $\left\Vert f\right\Vert _{L(\infty,\infty)}<\infty,$ where%
\[
\left\Vert f\right\Vert _{L(\infty,\infty)}=\sup_{t>0}\{f^{\ast\ast
}(t)-f^{\ast}(t)\},
\]
and%
\[
f^{\ast\ast}(t)=\frac{1}{t}\int_{0}^{t}f^{\ast}(s)ds.
\]
Then (cf. \cite[Lemma 3]{jn}, and also \cite[Theorem 4.1, pag 209]{torchinsky}
for a more detailed proof),

\begin{theorem}
\label{teodejn}Let $1<p<\infty.$ Suppose that $f\in JN_{p}(Q_{0}),$ then
$f-f_{Q_{0}}\in L(p,\infty)(Q_{0}),$ and there exists a constant
$A(p,Q_{0},n)$ such that
\[
\left\Vert f-f_{Q_{0}}\right\Vert _{L(p,\infty)(Q_{0})}\leq A(p,Q_{0}%
,n)JN_{p}(f,Q_{0}).
\]
In particular,%
\[
f-f_{Q_{0}}\in%
{\displaystyle\bigcap\limits_{r<p}}
L^{r}(Q_{0}).
\]

\end{theorem}

The limiting condition defining $JN_{p}(Q_{0})$ when $p=\infty$
corresponds\footnote{The $JN_{\infty}(Q_{0})$ condition would read%
\[
\sup_{\{Q_{i}\}_{i}\in P(Q_{0})}\frac{1}{\left\vert Q_{i}\right\vert }%
\int_{Q_{i}}\left\vert f-f_{Q_{i}}\right\vert dx<\infty.
\]
} to $BMO,$ and in this case Theorem \ref{teodejn} corresponds to a version of
the well known John-Nirenberg inequality \cite{jn}.

In the one dimensional case, Garsia and Rodemich \cite{garro} improved on
Theorem \ref{teodejn}. To formulate the Garsia and Rodemich result it will be
convenient to introduce a different scale of spaces which we shall term
Garsia-Rodemich spaces. It will be useful for later use to give the relevant
definitions in the $n-$dimensional case.

Let $Q_{0}\subset\mathbb{R}^{n}$ be a fixed cube$,$ let $1\leq p<\infty,$ and
let $p^{\prime}$ be defined by $\frac{1}{p}+\frac{1}{p^{\prime}}=1.$ The
Garsia-Rodemich spaces $GaRo_{p}(Q_{0})$ are defined as follows. We shall say
that $f\in GaRo_{p}(Q_{0}),$ if and only if $f\in L^{1}(Q_{0}),$ and $\exists
C>0$ such that for all $\{Q_{i}\}_{i\in N}$ $\in P(Q_{0})$ we have
\begin{equation}%
{\displaystyle\sum\limits_{i}}
\frac{1}{\left\vert Q_{i}\right\vert }\int_{Q_{i}}\int_{Q_{i}}\left\vert
f(x)-f(y)\right\vert dxdy\leq C\left(
{\displaystyle\sum\limits_{i}}
\left\vert Q_{i}\right\vert \right)  ^{1/p^{\prime}}.\label{launo}%
\end{equation}
We let%
\[
GaRo_{p}(f,Q_{0})=\inf\{C>0:\text{such that (\ref{launo}) holds}\}.
\]
Then we have (cf. \cite{garro})

\begin{theorem}
\footnote{Here it seems appropriate to bring up the following. In his paper
\cite{dy}, Dyson writes \textquotedblleft Professor Littlewood, when he makes
use of an algebraic identity always saves himself the trouble of proving it;
he maintains that an identity, if true, can be verified in a few lines by
anybody obtuse enough to feel the need of verification. My object in the
following pages is to confute this assertion". It is left to reader to decide
if the author of the present paper is demonstrating his own obtuseness..}%
\label{garsiarodemich}Let $1<p<\infty,$ and let $Q_{0}=I=[0,1].$ Then, as sets%
\[
GaRo_{p}(I)=L(p,\infty)(I).
\]

\end{theorem}

\begin{remark}
The elementary proof of the embedding $L(p,\infty)\subset GaRo_{p}$ outlined
in \cite{garro} works in $n-$dimensions and actually shows that (cf. Theorem
\ref{teomarkao} part (ii), below)%
\begin{equation}
GaRo_{p}(f,I)\leq\frac{p}{p-1}2\left\Vert f\right\Vert _{L(p,\infty)(I)}%
^{\ast}. \label{latelma}%
\end{equation}

\end{remark}

By Theorem\footnote{The fact that the containment is strict was shown in
\cite{aalto}.} \ref{teodejn} we have
\[
JN_{p}(I)\subsetneq L(p,\infty)(I),
\]
therefore by Theorem \ref{garsiarodemich} (cf. Section \ref{sec:garro} below
for a direct proof of the $n$ dimensional case) it follows that
\begin{equation}
JN_{p}(I)\subsetneq GaRo_{p}(I). \label{inmediata}%
\end{equation}
In conclusion, Theorem \ref{garsiarodemich} not only improves on Theorem
\ref{teodejn} in the one dimensional case, but also gives us an interesting
characterization of the Marcinkiewicz $L(p,\infty)(I)$ spaces$,1<p<\infty.$
Unfortunately, one part of the proof of Theorem \ref{garsiarodemich} uses a
non-trivial rearrangement inequality, also due to Garsia-Rodemich
\cite{garro}, which is only proved there in the one dimensional\footnote{See
also \cite{mamiast} for related inequalities.} case.

In \cite{garro}, the authors briefly suggest a possible different method to
prove Theorem \ref{garsiarodemich} in $n-$dimensions, and without dimensional
constants, but no details are provided\footnote{From \cite[page 115]{garro}:
\textquotedblleft We wish to point out also that using the Martingale
techniques of \cite{garsia} a proof of Theorem \ref{garsiarodemich} can be
obtained quite directly and without dimensional constants." We hope to follow
up this suggestion elsewhere.}. In this note we give a new proof Theorem
\ref{garsiarodemich} that is valid in $n$ dimensions (cf. Theorem
\ref{teomarkao} below) . Our approach is different from the one given in
\cite{garro}, and does not use martingale techniques. Instead, our method is
ultimately based on Calder\'{o}n-Zygmund type decompositions, following
classical ideas\footnote{It has the drawback of containing $n-$dimensional
constants.} in \cite{bds}.

As we shall see (cf. Section \ref{sec:garro}) the verification that the
John-Nirenberg conditions are stronger than the Garsia-Rodemich conditions
(e.g. (\ref{inmediata})) is immediate. Therefore, the crucial aspect of this
approach to the John-Nirenberg theorem is the fact that the Garsia-Rodemich
spaces are the same as the Marcinkiewicz $L(p,\infty)$ spaces! This clarifies
the self improvement results of John-Nirenberg. Moreover, these ideas could
potentially be useful in the investigation of related issues, e.g. the
dimensional constants involved in the John-Nirenberg embeddings (cf.
\cite{cwikel}).

Now the classical definitions of the $L(p,\infty)$ spaces are given in terms
of growth conditions on rearrangements or distribution functions (cf.
\cite{bs}, \cite{calderon}, \cite{oklander}, \cite{cwikelnilsson}, \cite{pus},
etc.). The case $p=\infty,$ which corresponds to $L(\infty,\infty)$ ("the
rearrangement invariant hull of $BMO$", cf. \cite{bds}), also admits a similar
characterization through the use of the oscillation operator $f^{\ast\ast
}-f^{\ast},$ and indeed one can find a characterization of all the
$L(p,\infty)$ spaces, $p\in(1,\infty],$ in the same fashion, namely%
\[
\left\Vert f\right\Vert _{L(p,\infty)}^{\#}=\sup_{s}\{(f^{\ast\ast}%
(s)-f^{\ast}(s))s^{1/p}\}<\infty.
\]
This characterization, while extremely useful in many problems (cf. \cite{bs},
\cite{corita}) is not always easy to implement, and does not reflect
immediately the self improvement\footnote{Note however that $(f^{\ast\ast
}(t)-f^{\ast}(t))=t\frac{d}{dt}(-f^{\ast\ast}(t)).$} of the Garsia-Rodemich
construction. In this direction, we found a different characterization of
$L(\infty,\infty),$ which gives an implicit differential inequality reflecting
the exponential decay of the distribution function of elements of
$L(\infty,\infty),$ via the use of distribution functions (cf. Section
\ref{sec:marcin} below)

\begin{theorem}
\label{teoweakinfty}Let $\left(  \Omega,\mu\right)  $ be a measure space.
Then, $f\in L(\infty,\infty):=L(\infty,\infty)\left(  \Omega\right)  $ if and
only if there exists $C>0$ such that for all $t>0,$
\begin{equation}
\int_{t}^{\infty}\lambda_{f}(s)ds\leq C\lambda_{f}(t), \label{form1}%
\end{equation}
and
\[
\left\Vert f\right\Vert _{L(\infty,\infty)}^{\#\#}:=\inf\{C:\text{such that
(\ref{form1}) holds\}}=\left\Vert f\right\Vert _{L(\infty,\infty)}.
\]

\end{theorem}

This characterization gives immediately the exponential integrability of
functions in $L(\infty,\infty),$ via the implicit differential inequality
(\ref{form1}). In fact, it is also welcome that there is similar
characterization for all $L(p,\infty)$ spaces, $1<p<\infty.$

\begin{theorem}
\label{teoweakp}Let $1<p<\infty.$ Let $\left(  \Omega,\mu\right)  $ be a
measure space. Then,%
\begin{align}
L(p,\infty)  &  :=L(p,\infty)\left(  \Omega\right)  =\{f\in L_{loc}^{1}\left(
\Omega\right)  :\left\Vert f\right\Vert _{L(p,\infty)}^{\ast}=\sup
_{s}\{f^{\ast}(s)s^{1/p}\}<\infty\}\label{lalapa}\\
&  =\{f\in L_{loc}^{1}\left(  \Omega\right)  :\left\Vert f\right\Vert
_{L(p,\infty)}=\sup_{s}\{f^{\ast\ast}(s)s^{1/p}\}<\infty\},\nonumber
\end{align}
coincides with the set of all $f$ such that $f^{\ast\ast}(\infty)=0,$ and%
\[
\left\Vert f\right\Vert _{L(p,\infty)}^{\#}=\sup_{s>0}\{(f^{\ast\ast
}(s)-f^{\ast}(s))s^{1/p}\}<\infty,
\]
which in turn coincides with the set of all $f$ such that $f^{\ast\ast}%
(\infty)=0,$ and%
\begin{equation}
\left\Vert f\right\Vert _{L(p,\infty)}^{\#\#}=\sup_{t>0}\{\frac{1}{\left(
\lambda_{f}(t)\right)  ^{1-1/p}}\int_{t}^{\infty}\lambda_{f}(s)ds\}<\infty.
\label{demostracion}%
\end{equation}

\end{theorem}

If one combines (\ref{demostracion}) with the usual definition of the spaces
$L(p,\infty)$ (cf. (\ref{lalapa})), one readily obtains a known
characterization of the $L(p,\infty)$ spaces which was apparently first given
by O'Neil \cite{oneil}.

\begin{corollary}
\label{elcorooneil}Let $1<p<\infty,$ then%
\begin{equation}
\left\Vert f\right\Vert _{L(p,\infty)}^{\ast}\sim\inf\{C^{1/p}:\int
_{t}^{\infty}\lambda_{f}(s)ds\leq Ct^{1-p}\}. \label{formaoneil}%
\end{equation}

\end{corollary}

\begin{remark}
One difference between (\ref{demostracion}) and (\ref{formaoneil}) is given by
the fact that the former also works in the case $p=\infty.$ Both formulations
can be extended to more general Marcinkiewicz spaces, $M_{\phi}$, where $\phi$
is a concave function. In particular, we refer to \cite{oneil} for the
corresponding theory of generalized Marcinkiewicz spaces $M_{\phi}$ defined
via (\ref{formaoneil})$.$
\end{remark}

In his expansive work \cite{oneil}, O'Neil used the formulae (\ref{formaoneil}%
) to study tensor products of $L(p,q)$ spaces (cf. also \cite{astashkin} and
\cite{tensor}). The space $L(\infty,\infty)$ was introduced later (cf.
\cite{bds}), and consequently was not considered in \cite{oneil}. In the last
section of this paper we give an application of (\ref{form1}) to show that
(cf. Theorem \ref{teotensorial} in Section \ref{sec:final} below)%
\begin{equation}
L(\infty,\infty)(\Omega_{1})\otimes L^{\infty}(\Omega_{2})\subset
L(\infty,\infty)(\Omega_{1}\times\Omega_{2}). \label{laidea}%
\end{equation}
While we think that (\ref{laidea}) could be useful in establishing other
embeddings of tensor products involving $L(\infty,\infty)$, such an
undertaking falls outside the scope of this note.

In conclusion, we should mention that this paper is part of series of papers
by the author on $BMO$, self improvement and interpolation, that go back at
least to \cite{lund}, \cite{fenicae1}, \cite{fenicae2}, with the most recent
opus being \cite{corita}, to which we refer for background information and
further references.

\section{John-Nirenberg spaces and Garsia-Rodemich spaces\label{sec:garro}}

It is easy to see the connection of the John-Nirenberg spaces with $BMO.$ Fix
a cube $Q_{0}\subset R^{n}$ and let%
\[
\left\Vert f\right\Vert _{BMO(Q_{0})}=\sup\{\frac{1}{\left\vert Q\right\vert
}\int_{Q}\left\vert f-f_{Q}\right\vert dx:Q\text{ subcube of }Q_{0}\}.
\]
Then, for $1\leq p<\infty$
\[
JN_{p}(f,Q_{0})\leq\left\Vert f\right\Vert _{BMO(Q_{0})}\left\vert
Q_{0}\right\vert ^{1/p}.
\]
Indeed, if $\{Q_{i}\}_{i\in N}\in P(Q_{0}),$ then we clearly have%
\begin{align*}
\left\{
{\displaystyle\sum\limits_{i}}
\left\vert Q_{i}\right\vert \left(  \frac{1}{\left\vert Q_{i}\right\vert }%
\int_{Q_{i}}\left\vert f-f_{Q_{i}}\right\vert dx\right)  ^{p}\right\}  ^{1/p}
&  \leq\left\{
{\displaystyle\sum\limits_{i}}
\left\vert Q_{i}\right\vert \left(  \left\Vert f\right\Vert _{BMO(Q_{0}%
)}\right)  ^{p}\right\}  ^{1/p}\\
&  \leq\left\Vert f\right\Vert _{BMO(Q_{0})}\left\vert Q_{0}\right\vert
^{1/p}.
\end{align*}

The purpose of this section is to prove the following

\begin{theorem}
\label{teomarkao}Let $1<p<\infty,$ and let $Q_{0}\subset R^{n}$ be a fixed
cube. Then

(i) $JN_{p}(Q_{0})\subset GaRo_{p}(Q_{0}),$ in fact%
\begin{equation}
GaRo_{p}(f,Q_{0})\leq2JN_{p}(f,Q_{0}). \label{laquesigue}%
\end{equation}

(ii) $GaRo_{p}(Q_{0})=L(p,\infty)(Q_{0}),$ in fact we have%
\begin{align*}
GaRo_{p}(f,Q_{0})  &  \leq\frac{2p}{p-1}\left\Vert f\right\Vert _{L(p,\infty
)}^{\ast},\text{ }\\
\sup_{t}t^{1/p}\left(  f^{\ast\ast}(t)-f^{\ast}(t)\right)   &  \leq
2^{n/p^{\prime}+1}GaRo_{p}(f,Q_{0})+\left(  \frac{4}{\left\vert Q_{0}%
\right\vert }\right)  ^{1/p^{\prime}}\left\Vert f\right\Vert _{L^{1}}.
\end{align*}

\end{theorem}

\begin{proof}
(i). Suppose that $\{Q_{i}\}_{i\in N}\in P(Q_{0}).$ Then for all $Q_{i}$
$,i\in N,$ we have,%
\begin{align*}
\int_{Q_{i}}\int_{Q_{i}}\left\vert f(x)-f(y)\right\vert dxdy  &  \leq
\int_{Q_{i}}\int_{Q_{i}}\left\vert f(x)-f_{Q_{i}}\right\vert dxdy+\int_{Q_{i}%
}\int_{Q_{i}}\left\vert f_{Q_{i}}-f(y)\right\vert dxdy\\
&  =2\left\vert Q_{i}\right\vert \int_{Q_{i}}\left\vert f-f_{Q_{i}}\right\vert
dx.
\end{align*}
Therefore,%
\begin{align*}%
{\displaystyle\sum\limits_{i}}
\frac{1}{\left\vert Q_{i}\right\vert }\int_{Q_{i}}\int_{Q_{i}}\left\vert
f(x)-f(y)\right\vert dxdy  &  \leq2%
{\displaystyle\sum\limits_{i}}
\int_{Q_{i}}\left\vert f-f_{Q_{i}}\right\vert dx\\
&  =2%
{\displaystyle\sum\limits_{i}}
\left\vert Q_{i}\right\vert ^{1/p^{\prime}}\left\vert Q_{i}\right\vert
^{1/p}\frac{1}{\left\vert Q_{i}\right\vert }\int_{Q_{i}}\left\vert f-f_{Q_{i}%
}\right\vert dx\\
&  \leq2\left(
{\displaystyle\sum\limits_{i}}
\left\vert Q_{i}\right\vert \right)  ^{1/p^{\prime}}\left\{
{\displaystyle\sum\limits_{i}}
\left\vert Q_{i}\right\vert \left(  \frac{1}{\left\vert Q_{i}\right\vert }%
\int_{Q_{i}}\left\vert f-f_{Q_{i}}\right\vert dx\right)  ^{p}\right\}  ^{1/p},
\end{align*}
and (\ref{laquesigue}) follows.

(ii). We show first that $L(p,\infty)(Q_{0})\subset GaRo_{p}(Q_{0}).$ Let
$\{Q_{i}\}_{i\in N}\in P(Q_{0}),$ then%
\begin{align*}%
{\displaystyle\sum\limits_{i}}
\frac{1}{\left\vert Q_{i}\right\vert }\int_{Q_{i}}\int_{Q_{i}}\left\vert
f(x)-f(y)\right\vert dxdy  &  \leq%
{\displaystyle\sum\limits_{i}}
\frac{1}{\left\vert Q_{i}\right\vert }\int_{Q_{i}}\int_{Q_{i}}\left(
\left\vert f(x)\right\vert +\left\vert f(y)\right\vert \right)  dxdy\\
&  \leq2\int_{\cup Q_{i}}\left\vert f(x)\right\vert dx\\
&  \leq2\int_{0}^{\sum_{i}\left\vert Q_{i}\right\vert }f^{\ast}(t)dt\\
&  \leq2\left\Vert f\right\Vert _{L(p,\infty)}^{\ast}\int_{0}^{\sum
_{i}\left\vert Q_{i}\right\vert }t^{-1/p}dt\\
&  =\frac{2p}{p-1}\left\Vert f\right\Vert _{L(p,\infty)}^{\ast}\left(
\sum_{i}\left\vert Q_{i}\right\vert \right)  ^{1/p^{\prime}}.
\end{align*}
Consequently,%
\[
GaRo_{p}(f,Q_{0})\leq\frac{2p}{p-1}\left\Vert f\right\Vert _{L(p,\infty
)}^{\ast}.
\]
To show the remaining inclusion, $GaRo_{p}(Q_{0})\subset L(p,\infty)(Q_{0}),$
we argue as in \cite[Chapter 5]{bs}. We provide all the details for the sake
of completeness.

To show that a function $f$ belongs to $L(p,\infty)(Q_{0})$ it is equivalent
to show that $\left\vert f\right\vert \in L(p,\infty)(Q_{0}),$ therefore,
since%
\[
GaRo_{p}(\left\vert f\right\vert ,Q_{0})\leq GaRo_{p}(f,Q_{0}),
\]
to show that $f\in GaRo_{p}(Q_{0})$ belongs to $L(p,\infty)(Q_{0}),$ we can
assume without loss that $f\geq0.$ Let $f\in GaRo_{p}(Q_{0}),$ $f\geq0.$ Fix
$t>0,$ such that $t<\left\vert Q_{0}\right\vert /4,$ and let $E=\{x\in
Q_{0}:f(x)>f^{\ast}(t)\}.$ By definition, $\left\vert E\right\vert \leq
t<\left\vert Q_{0}\right\vert /4,$ consequently, we can find a relatively open
subset of $Q_{0},$ $\Omega,$ say$,$ such that $E\subset\Omega$ and $\left\vert
\Omega\right\vert \leq2t\leq\left\vert Q_{0}\right\vert /2.$ By \cite[Lemma
7.2, page 377]{bs} we can find a sequence of cubes $\{Q_{i}\}_{i\in N},$ with
pairwise disjoint interiors, such that:%
\begin{align*}
(i)\text{ \ \ }\left\vert \Omega\cap Q_{i}\right\vert  &  \leq\frac{1}%
{2}\left\vert Q_{i}\right\vert \leq\left\vert \Omega^{c}\cap Q_{i}\right\vert
,\text{ }i=1,2...\\
(ii)\text{ \ \ \ }\Omega &  \subset%
{\displaystyle\bigcup\limits_{i\in N}}
Q_{i}\subset Q_{0}\\
(iii)\text{ \ \ }\left\vert \Omega\right\vert  &  \leq%
{\displaystyle\sum\limits_{i\in N}}
\left\vert Q_{i}\right\vert \leq2^{n+1}\left\vert \Omega\right\vert .
\end{align*}
Now, to estimate $t^{1/p}\left(  f^{\ast\ast}(t)-f^{\ast}(t)\right)  $, it
will be more convenient, by homogeneity, to consider $t\left(  f^{\ast\ast
}(t)-f^{\ast}(t)\right)  $ first$.$ Then, we have%
\begin{align*}
t\left(  f^{\ast\ast}(t)-f^{\ast}(t)\right)   &  =\int_{E}\{f(x)-f^{\ast
}(t)\}dx\\
&  \leq%
{\displaystyle\sum\limits_{i\in N}}
\int_{E\cap Q_{i}}\{f(x)-f^{\ast}(t)\}dx\\
&  =%
{\displaystyle\sum\limits_{i\in N}}
\left(  \int_{E\cap Q_{i}}\{f(x)-f_{Q_{i}}\}dx+\left\vert E\cap Q_{i}%
\right\vert \{f_{Q_{i}}-f^{\ast}(t)\}\right) \\
&  \leq%
{\displaystyle\sum\limits_{i\in N}}
\left(  \int_{Q_{i}}\{f(x)-f_{Q_{i}}\}dx+\left\vert E\cap Q_{i}\right\vert
\{f_{Q_{i}}-f^{\ast}(t)\}\right) \\
&  =(I)+(II).
\end{align*}
Let $J=\{i:f_{Q_{i}}>f^{\ast}(t)\},$ then%
\begin{align*}
(II)  &  =%
{\displaystyle\sum\limits_{i\in N}}
\left\vert E\cap Q_{i}\right\vert \{f_{Q_{i}}-f^{\ast}(t)\}\\
&  \leq%
{\displaystyle\sum\limits_{i\in J}}
\left\vert E\cap Q_{i}\right\vert \{f_{Q_{i}}-f^{\ast}(t)\}\\
&  \leq%
{\displaystyle\sum\limits_{i\in J}}
\left\vert \Omega\cap Q_{i}\right\vert \{f_{Q_{i}}-f^{\ast}(t)\}\\
&  \leq%
{\displaystyle\sum\limits_{i\in J}}
\left\vert \Omega^{c}\cap Q_{i}\right\vert \{f_{Q_{i}}-f^{\ast}(t)\}\\
&  =%
{\displaystyle\sum\limits_{i\in J}}
\int_{\Omega^{c}\cap Q_{i}}\{f_{Q_{i}}-f^{\ast}(t)\}dx\\
&  \leq%
{\displaystyle\sum\limits_{i\in J}}
\int_{\Omega^{c}\cap Q_{i}}\{f_{Q_{i}}-f(x)\}dx\text{ (since }\Omega
^{c}\subset E^{c}\text{)}\\
&  \leq%
{\displaystyle\sum\limits_{i\in J}}
\int_{Q_{i}}\left\vert f_{Q_{i}}-f(x)\right\vert dx\\
&  \leq%
{\displaystyle\sum\limits_{i\in J}}
\frac{1}{\left\vert Q_{i}\right\vert }\int_{Q_{i}}\int_{Q_{i}}\left\vert
f(y)-f(x)\right\vert dxdy\\
&  \leq GaRo_{p}(f,Q_{0})\left(
{\displaystyle\sum\limits_{i\in N}}
\left\vert Q_{i}\right\vert \right)  ^{1/p^{\prime}}.
\end{align*}
Likewise,%
\begin{align*}
(I)  &  =%
{\displaystyle\sum\limits_{i\in N}}
\int_{Q_{i}}\{f(x)-f_{Q_{i}}\}dx\\
&  =%
{\displaystyle\sum\limits_{i\in N}}
\frac{1}{\left\vert Q_{i}\right\vert }\int_{Q_{i}}\int_{Q_{i}}%
(f(x)-f(y))dxdy\\
&  \leq%
{\displaystyle\sum\limits_{i\in N}}
\frac{1}{\left\vert Q_{i}\right\vert }\int_{Q_{i}}\int_{Q_{i}}\left\vert
f(x)-f(y)\right\vert dxdy\\
&  \leq GaRo_{p}(f,Q_{0})\left(
{\displaystyle\sum\limits_{i\in N}}
\left\vert Q_{i}\right\vert \right)  ^{1/p^{\prime}}.
\end{align*}
Combining the inequalities we have obtained,%
\begin{align*}
t\left(  f^{\ast\ast}(t)-f^{\ast}(t)\right)   &  \leq2GaRo_{p}(f,Q_{0})\left(
%
{\displaystyle\sum\limits_{i\in N}}
\left\vert Q_{i}\right\vert \right)  ^{1/p^{\prime}}\\
&  \leq2GaRo_{p}(f,Q_{0})(2^{n+1})^{1/p^{\prime}}2^{-1/p^{\prime}%
}t^{1/p^{\prime}}.
\end{align*}
Therefore,%
\[
\sup_{t\leq\left\vert Q_{0}\right\vert /4}t^{1/p}\left(  f^{\ast\ast
}(t)-f^{\ast}(t)\right)  \leq2^{n/p^{\prime}+1}GaRo_{p}(f,Q_{0}).
\]
To deal with $t>\left\vert Q_{0}\right\vert /4,$ we note that $t\left(
f^{\ast\ast}(t)-f^{\ast}(t)\right)  =\int_{f^{\ast}(t)}^{\infty}\lambda
_{f}(s)ds\leq\int_{0}^{\infty}\lambda_{f}(s)ds=\left\Vert f\right\Vert
_{L^{1}};$ therefore,%
\begin{align*}
t^{1/p}\left(  f^{\ast\ast}(t)-f^{\ast}(t)\right)   &  \leq t^{-1/p^{\prime}%
}\left\Vert f\right\Vert _{L^{1}}\\
&  \leq\left(  \frac{4}{\left\vert Q_{0}\right\vert }\right)  ^{1/p^{\prime}%
}\left\Vert f\right\Vert _{L^{1}}.
\end{align*}
Thus,%
\[
\sup_{t}t^{1/p}\left(  f^{\ast\ast}(t)-f^{\ast}(t)\right)  \leq2^{n/p^{\prime
}+1}GaRo_{p}(f,Q_{0})+\left(  \frac{4}{\left\vert Q_{0}\right\vert }\right)
^{1/p^{\prime}}\left\Vert f\right\Vert _{L^{1}},
\]
and the desired result follows by Theorem \ref{teoweakp}.
\end{proof}

\section{Another characterization of the $L(p,\infty)$ spaces, $1<p\leq\infty
$\label{sec:marcin}}

The purpose of this section is to give a proof of Theorem \ref{teoweakinfty}
and Theorem \ref{teoweakp}.

We start with the former.

\begin{proof}
Let $f$ be such that there exists $C>0$ such that (\ref{form1}) holds for all
$t>0$. Then, we have%
\begin{align*}
\left(  f^{\ast\ast}(t)-f^{\ast}(t)\right)  t  &  =\int_{f^{\ast}(t)}^{\infty
}\lambda_{f}(s)ds\\
&  \leq C\lambda_{f}(f^{\ast}(t))\\
&  \leq Ct.
\end{align*}
Thus,%
\begin{align*}
\left\Vert f\right\Vert _{L(\infty,\infty)}  &  \leq\inf\{C:(\ref{form1}%
)\text{ holds}\}\\
&  =\left\Vert f\right\Vert _{L(\infty,\infty)}^{\#\#}.
\end{align*}
Conversely, suppose that $f\in L(\infty,\infty).$ Then, for all $t>0,$ we
have,%
\[
\int_{f^{\ast}(t)}^{\infty}\lambda_{f}(s)ds=\left(  f^{\ast\ast}(t)-f^{\ast
}(t)\right)  t\leq t\left\Vert f\right\Vert _{L(\infty,\infty)}.
\]
Therefore,%
\[
\int_{f^{\ast}(\lambda_{f}(t))}^{\infty}\lambda_{f}(s)ds\leq\lambda
_{f}(t)\left\Vert f\right\Vert _{L(\infty,\infty)}.
\]
Now, since $f^{\ast}(\lambda_{f}(t))\leq t,$ we have%
\[
\int_{t}^{\infty}\lambda_{f}(s)ds\leq\lambda_{f}(t)\left\Vert f\right\Vert
_{L(\infty,\infty)}.
\]
Consequently,%
\[
\left\Vert f\right\Vert _{L(\infty,\infty)}^{\#\#}\leq\left\Vert f\right\Vert
_{L(\infty,\infty)},
\]
concluding the proof.
\end{proof}

We proceed with the proof of Theorem \ref{teoweakp}.

\begin{proof}
We trivially have $\left\Vert f\right\Vert _{L(p,\infty)}^{\#}\leq\left\Vert
f\right\Vert _{L(p,\infty)}.$ Moreover, if $f^{\ast\ast}(\infty)=0,$ then%
\begin{align*}
f^{\ast\ast}(t)t^{1/p}  &  =t^{1/p}\int_{t}^{\infty}\left(  f^{\ast\ast
}(s)-f^{\ast}(s)\right)  \frac{ds}{s}\\
&  =t^{1/p}\int_{t}^{\infty}\left(  f^{\ast\ast}(s)-f^{\ast}(s)\right)
s^{1/p}s^{-1/p}\frac{ds}{s}\\
&  \leq t^{1/p}\left\Vert f\right\Vert _{L(p,\infty)}^{\#}\int_{t}^{\infty
}s^{-1/p}\frac{ds}{s}\\
&  =p\left\Vert f\right\Vert _{L(p,\infty)}^{\#}.
\end{align*}
Consequently,
\[
\left\Vert f\right\Vert _{L(p,\infty)}\leq p\left\Vert f\right\Vert
_{L(p,\infty)}^{\#}.
\]
The last part of the result follows exactly as the proof of Theorem
\ref{teoweakinfty} (the case $p=\infty).$ For example, from%
\[
\int_{t}^{\infty}\lambda_{f}(s)ds\leq\left\Vert f\right\Vert _{L(p,\infty
)}^{\#\#}\left(  \lambda_{f}(t)\right)  ^{1-1/p}%
\]
we get%
\[
\left(  f^{\ast\ast}(t)-f^{\ast}(t)\right)  t=\int_{f^{\ast}(t)}^{\infty
}\lambda_{f}(s)ds\leq\left\Vert f\right\Vert _{L(p,\infty)}^{\#\#}t^{1-1/p}%
\]
and therefore%
\[
\left\Vert f\right\Vert _{L(p,\infty)}^{\#}\leq\left\Vert f\right\Vert
_{L(p,\infty)}^{\#\#}.
\]
Conversely, for all $t>0,$%
\begin{align*}
t\left\Vert f\right\Vert _{L(p,\infty)}^{\#}  &  \geq tt^{1/p}\left(
f^{\ast\ast}(t)-f^{\ast}(t)\right) \\
&  \geq t^{1/p}\int_{f^{\ast}(t)}^{\infty}\lambda_{f}(s)ds.
\end{align*}
Thus,
\begin{align*}
\lambda_{f}(t)\left\Vert f\right\Vert _{L(p,\infty)}^{\#}  &  \geq\left(
\lambda_{f}(t)\right)  ^{1/p}\int_{f^{\ast}(\lambda_{f}(t))}^{\infty}%
\lambda_{f}(s)ds\\
&  \geq\left(  \lambda_{f}(t)\right)  ^{1/p}\int_{t}^{\infty}\lambda_{f}(s)ds,
\end{align*}
and the desired result follows.
\end{proof}

\begin{remark}
Observe that when $p=1,$ the previous considerations provide a
characterization of $L^{1},$ not of $L(1,\infty).$ Indeed, the corresponding
result for $p=1$ is%
\begin{align*}
\sup_{t}f^{\ast\ast}(t)t  &  =\sup_{t}\int_{0}^{t}f^{\ast}(s)ds=\left\Vert
f\right\Vert _{L^{1}}\\
&  =\sup_{t}\int_{t}^{\infty}\lambda_{f}(s)ds.
\end{align*}
In other words, $L^{1}$ is characterized by the condition%
\[
\sup_{t>0}\int_{t}^{\infty}\lambda_{f}(s)ds\leq C.
\]
Note that in this case $\left(  \lambda_{f}(t)\right)  ^{1-1/1}=1.$
\end{remark}

To conclude this section we prove Corollary \ref{elcorooneil}.

\begin{lemma}
Let $1<p<\infty.$ Then%
\[
\left\Vert f\right\Vert _{L(p,\infty)}^{\#\#}\approx\inf\{C^{1/p}:\int
_{t}^{\infty}\lambda_{f}(s)ds\leq Ct^{1-p}\}.
\]

\end{lemma}

\begin{proof}
Note that%
\begin{equation}
\left(  \lambda_{f}(t)\right)  ^{1/p}\leq\left\Vert f\right\Vert
_{L(p,\infty)}t^{-1}\Leftrightarrow\lambda_{f}(t)\leq\left\Vert f\right\Vert
_{L(p,\infty)}^{p}t^{-p}.\label{chiado}%
\end{equation}
Suppose that $f\in L(p,\infty).$ By the previous Theorem,
\begin{align*}
\int_{t}^{\infty}\lambda_{f}(s)ds &  \leq\left\Vert f\right\Vert
_{L(p,\infty)}^{\#\#}\left(  \lambda_{f}(t)\right)  ^{1-1/p}\\
&  \leq\left\Vert f\right\Vert _{L(p,\infty)}^{\#\#}\left\Vert f\right\Vert
_{L(p,\infty)}^{p(1-1/p)}t^{-p(1-1/p)}\text{ (by (\ref{chiado}))}\\
&  \leq C\left\Vert f\right\Vert _{L(p,\infty)}^{p}t^{1-p}.
\end{align*}
Conversely, suppose that%
\[
\int_{t}^{\infty}\lambda_{f}(s)ds\leq Ct^{1-p}.
\]
Then, since $\lambda_{f}$ decreases,
\begin{align*}
\lambda_{f}(t)^{1/p}\frac{t}{2} &  \leq\int_{t/2}^{t}\lambda_{f}%
(s)^{1/p}ds=\int_{t/2}^{t}\lambda_{f}(s)\lambda_{f}(s)^{1/p-1}ds\\
&  \leq\lambda_{f}(t/2)^{1/p-1}\int_{t/2}^{t}\lambda_{f}(s)ds\\
&  \leq\lambda_{f}(t/2)^{1/p-1}\int_{t/2}^{\infty}\lambda_{f}(s)ds\\
&  \leq\lambda_{f}(t/2)^{1/p-1}C2^{p-1}t^{1-p}%
\end{align*}
Therefore,%
\[
\lambda_{f}(t/2)^{-1/p+1}\lambda_{f}(t)^{1/p}\leq\tilde{C}t^{-p}%
\]
and, consequently,%
\[
\lambda_{f}(t)^{-1/p+1}\lambda_{f}(t)^{1/p}\leq\tilde{C}t^{-p}.
\]
The desired result follows.
\end{proof}

\section{Final Remarks\label{sec:final}}

\subsection{Products and tensor products with $L(\infty,\infty)$}

The new formulae we presented for the computation of the \textquotedblleft
norm" $\left\Vert {}\right\Vert _{L(\infty,\infty)}$ has several applications.
Here, following O'Neil \cite{oneil} (cf. also \cite{astashkin}, \cite{tensor}%
), we shall briefly consider tensor products with $L(\infty,\infty).$ It is
not our purpose to develop the most general results, but to give a flavor of
the ideas involved.

\begin{theorem}
\label{teotensorial}Let $(\Omega_{1},\mu_{1}),(\Omega_{2},\mu_{2}),$ be
measure spaces. Then,
\begin{equation}
L(\infty,\infty)(\Omega_{1})\otimes L^{\infty}(\Omega_{2})\subset
L(\infty,\infty)(\Omega_{1}\times\Omega_{2}), \label{tensorial}%
\end{equation}
with%
\[
\left\Vert f\otimes g\right\Vert _{L(\infty,\infty)(\Omega_{1}\times\Omega
_{2})}\leq\left\Vert f\right\Vert _{L(\infty,\infty)(\Omega_{1})}\left\Vert
g\right\Vert _{L^{\infty}(\Omega_{2})}.
\]

\end{theorem}

\begin{proof}
Let $f\in L(\infty,\infty)(\Omega_{1}),g\in L^{\infty}(\Omega_{2}).$ The
distribution function of $f\otimes g$ is computed in \cite[Lemma 7.1 (2), page
97]{oneil}%
\begin{equation}
\lambda_{f\otimes g}(z)=\int_{0}^{\infty}\lambda_{f}(\frac{z}{u}%
)d(-\lambda_{g}(u)),z>0. \label{fundamenta}%
\end{equation}
Therefore, on account that $g\in L^{\infty},$ we have%
\begin{equation}
\lambda_{f\otimes g}(z)=\int_{0}^{\left\Vert g\right\Vert _{L^{\infty}%
(\Omega_{2})}}\lambda_{f}(\frac{z}{u})d(-\lambda_{g}(u)). \label{fundamenta1}%
\end{equation}
Then, by Tonnelli's theorem, for all $t>0,$%
\begin{align*}
\int_{t}^{\infty}\lambda_{f\otimes g}(z)dz  &  =\int_{t}^{\infty}\int
_{0}^{\left\Vert g\right\Vert _{L^{\infty}(\Omega_{2})}}\lambda_{f}(\frac
{z}{u})d(-\lambda_{g}(u))dz\\
&  =\int_{0}^{\left\Vert g\right\Vert _{L^{\infty}(\Omega_{2})}}\int
_{t}^{\infty}\lambda_{f}(\frac{z}{u})dzd(-\lambda_{g}(u))\\
&  =\int_{0}^{\left\Vert g\right\Vert _{L^{\infty}(\Omega_{2})}}\int_{\frac
{t}{u}}^{\infty}\lambda_{f}(r)udrd(-\lambda_{g}(u))\\
&  \leq\left\Vert g\right\Vert _{L^{\infty}(\Omega_{2})}\left\Vert
f\right\Vert _{L(\infty,\infty)(\Omega_{1})}^{\#\#}\int_{0}^{\left\Vert
g\right\Vert _{L^{\infty}(\Omega_{2})}}\lambda_{f}(\frac{t}{u})d(-\lambda
_{g}(u))\\
&  =\left\Vert g\right\Vert _{L^{\infty}(\Omega_{2})}\left\Vert f\right\Vert
_{L(\infty,\infty)(\Omega_{1})}^{\#\#}\lambda_{f\otimes g}(t)\text{ (by
(\ref{fundamenta1})).}%
\end{align*}
Hence, by Theorem \ref{teoweakinfty},%
\[
\left\Vert f\otimes g\right\Vert _{L(\infty,\infty)(\Omega_{1}\times\Omega
_{2})}^{\#\#}\leq\left\Vert f\right\Vert _{L(\infty,\infty)(\Omega_{1}%
)}^{\#\#}\left\Vert g\right\Vert _{L^{\infty}(\Omega_{2})}.
\]
\end{proof}

\subsection{More Problems}

1. It seems to us that our approach to prove the second part of Theorem
\ref{teomarkao} can be modified to study the rearrangement inequality of
Garsia-Rodemich \cite[Theorem 7.3]{garro} in the $n-$dimensional case.

2. It would be of interest to follow up on the suggestion of Garsia-Rodemich
and prove a version of Theorem \ref{teomarkao} using the methods of
\cite{garsia}.

3. It would be of interest to complete the study of tensor products with
$L(\infty,\infty).$


\begin{thebibliography}{99}                                                                                               %


\bibitem {aalto}D. Aalto, L. Berkovits, O. E. Kansanen and H. Yue,
\textsl{John-Nirenberg lemmas for a doubling measure}. Studia Math.
\textbf{204} (2011), 21-37.

\bibitem {astashkin}S. V. Astashkin, \textsl{Tensor product in symmetric
function spaces}. Collect. Math. \textbf{48} (1997), 375--391.

\bibitem {bds}C. Bennett, R. DeVore and R. Sharpley, \textsl{Weak-}$L^{\infty
}$\textsl{ and }$BMO$. Ann. Math. \textbf{113} (1981), 601-611.

\bibitem {bs}C. Bennett and R. Sharpley, \textsl{Interpolation of operators}.
Academic Press, 1988.

\bibitem {calderon}A. P. Calder\'{o}n, \textsl{Spaces between }$L^{1}$\textsl{
and }$L^{\infty}$\textsl{ and the theorem of Marcinkiewicz}. Studia Math.
\textbf{26} (1966), 273-299.

\bibitem {cwikelnilsson}M. Cwikel and P. Nilsson, \textsl{Interpolation of
Marcinkiewicz spaces}. Math. Scand. \textbf{56} (1985), 29-42.

\bibitem {cwikel}M. Cwikel, Y. Sagher, P. Shvartsman, \textsl{A new look at
the John--Nirenberg and John--Str\"{o}mberg theorems for BMO. }J. Funct. Anal.
\textbf{263} (2012), 129-167.

\bibitem {dy}F. J. Dyson, \textsl{Some guesses in the theory of partitions}.
Eureka \textbf{8 }(1944), 10-15.

\bibitem {garsia}A. M. Garsia, \textsl{Martingale inequalities: Seminar notes
on recent progress}. Mathematics Lecture Notes Series. W. A. Benjamin, Inc.,
Reading, Mass.-London-Amsterdam, 1973.

\bibitem {garro}A. M. Garsia and E. Rodemich, \textsl{Monotonicity of certain
functional under rearrangements}. Ann. Inst. Fourier (Grenoble) \textbf{24}
(1974), 67-116.

\bibitem {jn}F. John and L. Nirenberg, \textsl{On functions of bounded mean
oscillation}. Comm. Pure Appl. Math. \textbf{14} (1961), 415-426.

\bibitem {mamiast}J. Martin and M. Milman, \textsl{Fractional Sobolev
Inequalities: Symmetrization, Isoperimetry and Interpolation}. Ast\'{e}risque
\textbf{366} (2014).

\bibitem {lund}M. Milman, \textsl{Rearrangements of }$BMO$\textsl{ functions
and interpolation}, in Lecture Notes in Mathematics 1070, pp 208-212.

\bibitem {fenicae1}M. Milman, \textsl{A note on Gehring's lemma}. Ann. Acad.
Sci. Fenn. \textbf{21} (1996), 389-398.

\bibitem {fenicae2}M. Milman, \textsl{A note on interpolation and higher
integrability}. Ann. Acad. Sci. Fenn. \textbf{23} (1998), 169-180.

\bibitem {corita}M. Milman, \textsl{BMO: Oscillations, self-improvement,
Gagliardo coordinate spaces and reverse Hardy inequalities}. arXiv:1505.02633,
to appear in Harmonic Analysis, Partial Differential Equations, Banach Spaces, and Operator Theory. 
Celebrating Cora Sadosky's Life. Volume 1 (Edited by S. Marcantognini, M. C. Pereyra, A. Stokolos
and W. Urbina), AWM-Springer Series.

\bibitem {tensor}M. Milman, \textsl{Some new function spaces and their tensor
products}. Bull. Australian Math. Soc. \textbf{19} (1978), 147 - 149.

\bibitem {oklander}E. T. Oklander, \textsl{Interpolaci\'{o}n, espacios de
Lorentz y teorema de Marcinkiewicz}. Cursos y Seminarios de Matem\'{a}ticas
\textbf{20} (1965), Univ. Buenos Aires, 1965.

\bibitem {oneil}R. O'Neil, \textsl{Integral transforms and tensor products on
Orlicz spaces and }$L(p,q)$\textsl{ spaces}. J. d'Analyse Math. \textbf{21}
(1968), 4-276.

\bibitem {pus}E. Pustylnik, \textsl{On some properties of generalized
Marcinkiewicz spaces}. Studia Math. \textbf{144} (2001), 227-243.

\bibitem {torchinsky}A. Torchinsky, \textsl{Real variable methods in harmonic
analysis}. Academic Press, 1986.
\end{thebibliography}
\end{document}